\documentclass[reqno,12pt]{amsart}

\usepackage{multirow}
\usepackage[centering,text={18cm,23cm},
		marginparwidth=20mm]{geometry}

\usepackage{amsmath}
\usepackage{amssymb}
\usepackage{graphicx}
\usepackage{fontenc}
\usepackage{amsthm}
\usepackage{enumerate}
\usepackage{bbold}
\usepackage{color}

\theoremstyle{plain}
\newtheorem{theorem}{Theorem}

\newtheorem{lemma}[theorem]{Lemma}

\numberwithin{equation}{section}
\newtheorem{thm}{Theorem}[section]

\newtheorem{prop}[thm]{Proposition}

\newcommand*{\Z}{\mathbb{Z}}

\newcommand*{\F}{\mathbb{F}}

\newcommand*{\PP}{\mathbb{P}}

\newcommand*{\Disc}{\textrm{Disc}}

\def\Cl{{\rm Cl}}

\def\Pic{{\rm Pic}}

\def\Disc{{\rm Disc}}

\def\rk{{\rm rk}}

\everymath{\displaystyle}

\begin{document}
\title[]{High $\ell$-torsion rank for class groups over function fields}%
\author{I.Setayesh, J.Tsimerman }

\begin{abstract}
 
 We prove that in the function field setting, $\ell$-torsion in the class groups of quadratic fields can be arbitrarily large. In fact, we explicitly produce a family
 whose $\ell$-rank growth matches the growth in the setting of genus theory, which might be best possible.  We do this
by specifically focusing on the Artin-Schreir curves $y^2=x^q-x$.
 
\end{abstract}
\maketitle

\section{Introduction}

Given a number field $K$, we denote the class group of $K$ by $\Cl(K)$. For a fixed a prime number $\ell$ there is much work studying the behaviour of the torsion subgroup
$\Cl(K)[\ell]$ as $K$ varies in a natural family. Our state of knowledge is however still very limited, even for the case of imaginary quadratic fields. On the one hand, it is in
general still an open question to provide unconditional upper bounds which improve on the trivial bound $|\Cl(K)[\ell]|\leq |\Cl(K)|$. On the other hand, it is an open question
whether the rank of $\Cl(K)[\ell]$ is even unbounded. While we have nothing to say on this matter for number fields, the purpose of this paper is to resolve
the analogue of this latter question over function fields. This has been a folklore conjecture for some time, but it was conjectured by \v{C}esnavi\v{c}ius in \cite{C}, who proved certain cases of it.

We follow the usual analogy, in that we work with hyperelliptic curves $C$ over a finite field $k$, and treat $\Pic^0(C)(k)$ as the analogue of the class group $\Cl(K)$.

\begin{thm}\label{main}

Let $p\neq \ell$ be odd prime numbers. There is a sequence of hyperelliptic curves $C_i$ over $\F_p$ of increasing genus $g_i$ such that 
$$\frac{g_i \log_\ell p}{\log_p(g_i)}\cdot(1+o(1))\leq \rk \Pic^0(C_i)(\F_p)[\ell]\leq \frac{g_i \log_\ell p}{\log_p(g_i)}\cdot(2+o(1)).$$

\end{thm}

It is interesting to ask whether the rate of growth we obtain in the theorem should be best possible up to constants. Treating $p^{2g}$ as analogous 
to the Discriminant $\Disc_K$, Theorem \ref{main} obtains the order of growth $\frac{\log \Disc_K}{\log\log \Disc_K}$. 
In this vein, Brumer and Silverman\cite{BS} ask whether the 
best possible upper bound for the $\ell$-rank of the class group of a bounded degree number field $K$ is of order
$\frac{\log \Disc_K}{\log\log \Disc_K}$. It is easy to see using genus theory that this is true for 2-torsion in quadratic fields, for example. Nonetheless, to our knowledge there is no 
theoretical or statistical argument motivating this upper bound.

Finally, we remark that with some more work it should be possible to remove the factor of 2 between the upper and lower bounds in the theorem to obtain a precise asymptotic.

\subsection{Outline of the paper}

Our method of proof is to study the specific family of Artin-Schreir curves $C_q$ defined by  $y^2=x^q-x$ for $q$ a prime power of $p$. This curve is supersingular, and 
we can explicitly identify all of its eigenvalues, and pick out their multiplicity asymptotically. This is done in section 2. In section 3 we determine the structure of the Tate module
of these curves. Finally in section 4 we combine everything to prove the main theorem.

\section{Size of the Jacobian of $C_q$}
In this section we compute the number of $\F_q$-points of the Jacobian of the Artin-Schreier curve $$C_q : y^2=x^q-x$$ This
hyperelliptic curve is ramified over the $q+1$ points $\PP^1(\F_q)$ and so it has genus $\frac12(q-1)$. Set $q^*=(-1)^{\frac{q-1}{2}}q$. We let $\alpha_1,\dots,\alpha_{2g}$ denote the Frobenius eigenvalues of $C_q$. 

\begin{lemma}\label{eigencq}
The multiset $\{\alpha_1,\dots,\alpha_{2g}\}$ of $F_q$-Frobenius eigenvalues of $C_q$ consists of $g$ repetitions of the set 
$\{\sqrt{q^*},-\sqrt{q^*}\}$. 
\end{lemma}

\begin{proof}

We begin by computing the number of points of $C_q$ over $\F_q$ and $\F_{q^2}$. First, $C_q$ maps to 
$\PP^1$ and is ramified precisely over $\PP^1(\F_q)$. Thus, $C_q(\F_q)$ is in bijection with $\PP^1(\F_q)$, and is therefore of size $q+1$.
It follows that $\sum_{i=1}^{2g}\alpha_i = 0$.

Next, note that we may write $\F_{q^2}=\F_q[\beta]$ where $\beta^2=\epsilon, \epsilon\in\F_q$. Since the multiplicative group
of a finite field is cyclic, we see that $\beta$ is a square in $\F_{q^2}$ iff $q\equiv3\mod 4$. 
If we set $x=a+b\epsilon$ with $a,b\in\F_q$ and $b\neq 0$, then the equation of the curve becomes $y^2=b(\beta^q-\beta)=-2b\beta$. 
Since $-2$ is a square in $\F_{q^2}$, we see that every $x\in\F_{q^2}\backslash\F_q$ gives 0 or 2 points depending on whether
$q$ is 1 or $3$ mod 4. Thus, $\sharp C_q(\F_{q^2}) = q^2+1-(q-1)q^*=q^2+1-2gq^*$. It follows that 
$\sum_{i=1}^{2g}\alpha_i^2=2gq^*$. Since all the $\alpha_i^2$ are $q^2$-Weil numbers, it follows that $\alpha_i^2=q^*$ for
every $i$. Finally, since the sum of all the $\alpha_i$ is 0, it follows that half of them are $\sqrt{q^*}$ and half are $-\sqrt{q^*}$, as
desired.
\end{proof}

\bigskip

Next, we determine the Frobenius eigenvalues of $C_{q}$ viewed as a curve over smaller fields. To this end, we let $q_0\cong 1\mod4$ be a prime power that is a square, $q=q_0^m$, and consider the eigenvalues of $C_q$ over $\F_{q_0}$. We know by lemma \ref{eigencq} that the eigenvalues are all $2m$'th roots of unity times $q_0^{1/2}$. Since primitive roots of unity of any order are all Galois conjugate, for each $d|2m$ we define $a_d$ to be the number of copies of $q_0^{1/2}\mu_{d}$ where $\mu_d$ denotes the set of primitive $d$'th roots of unity. Our goal is to show that for large enough $q$, all the $a_d$ are positive.

\bigskip

To that end, we  define $f(d,s):=\sum_{x\in\mu_d} x^s$. Note that $f(d,s)$ is a multiplicative function in $d$, and for a prime number $r$  we have
$$f(r^a,r^b)=\begin{cases} 1 & a=0\\ 0 & a>b+1\\ -r^{a-1} & a=b+1\\ (r-1)r^{a-1} & 0<a\leq b \end{cases}.$$

\begin{lemma}\label{pinversion}

Consider a prime $r$ and a positive integer $k$. Then for $0\leq d,e\leq k$, we have
$$\sum_{0\leq i\leq k}f(r^d,r^i)f(r^{k-i},r^{k-e}) = p^k\delta_{e=d}.$$

\end{lemma}

\begin{proof}

This is a straightforward computation using geometric series. We divide into cases.

\begin{enumerate}

\item $d=0,e=0$. The identity becomes $\sum_{j|r^k}\phi(j)=p^k$. 

\item $d=0,e>0$. The identity becomes
$$-r^{k-e} + (r-1)\sum_{e\leq i<k}r^{k-i-1} + 1=0.$$

\item $d>0,e=0$. The identity becomes 
$$-r^{k-1}(r-1)+\sum_{d\leq i<k}r^{k-2-i+d}(r-1)^2 + r^{d-1}(r-1)=0.$$

\item $d=e>0$. The identity becomes
$$r^{k-1}+(r-1)^2\sum_{d\leq i<k} r^{k-2-i+d}+r^{d-1}(r-1)=r^k$$

\item $d>e>0$. The identity becomes
$$-r^{k-1}(r-1)+(r-1)^2\sum_{d\leq i<k} r^{k-2-i+d}+r^{d-1}(r-1)=0$$

\item $e>d>0$. The identity becomes$$-r^{k-1+d-e}(r-1)+(r-1)^2\sum_{e\leq i<k} r^{k-2-i+d}+r^{d-1}(r-1)=0$$

\end{enumerate}

\end{proof}

For each $s|2m$ we have the following identity:

$$\frac{\#C(\F_{q_0^s}) - (q_0^s+1)}{q_0^{s/2}}=\sum_{d|2m}a_df(d,s)$$ from which we obtain the following:

\begin{itemize}

\item For $s|2m$ and $s\neq 2m$, we have $$|\sum_{d|2m}a_df(d,s)|< q_0^{s/2}+1\leq q^{\frac12}+1.$$

\item Since $q\cong 1\mod 4$, by lemma \ref{eigencq} we have $\sum_{d|2m}a_df(d,2m) = q-1$.

\end{itemize}

We now prove our main result for this section:

\begin{prop}\label{adappears}

We have the inequality $$|ma_d-(q-1)|\leq(m-1)(q^{\frac12}+1).$$

\end{prop}

\begin{proof}

For $d|2m$ set $c_s(d) = \prod_{r^k||2m} f(r^k/(s,r^k),r^k/(d,r^k))$. It follows from lemma \ref{pinversion} that

$$\sum_{s|2m}c_s(d)\sum_{d|2m}a_df(d,s) = ma_d.$$

Thus, we see that $$ma_d=q-1 + \sum_{s|2m,s<2m} c_s(d)\sum_{d|2m}a_df(d,s).$$ From our inequalities, we thus have

$$|ma_d-(q-1)| \leq (q^{\frac12}+1)\cdot(\sum_{s|2m}|c_s(d)| -1) \leq(m-1)(q^{\frac12}+1)$$

as desired.

\end{proof}

\section{Structure of the Tate module}

Let $J_q$ denote the Jacobian of $C_q$. The next lemma shows that the torsion of $J_q$ is of a particularly simple form.

\begin{lemma}\label{tate}

Assume that $m$ is divisible by $\ell-1$ and co-prime to $\ell$. Let $M$ be the power of $(x-1)$ dividing the characteristic polynomial of the Frobenius $F_{q_0}$ modulo $\ell$ for $C_q$. Then 
$$|\#J_q(\F_{q_0})[\ell]| = \ell^M.$$ 

\end{lemma}

\begin{proof}

Let $T_\ell$ denote the $\ell$-adic Tate module of $J_q$. Let $a\in\mu_{\ell-1}(\Z_\ell)$ be such that $aq_0^{\frac12}\equiv 1\mod{\ell}$. Let $F_0=\frac{F}{aq_0^{\frac12}}$. 
Note that $F_0\equiv F\mod\ell$ so that $$\#J_q(\F_{q_0})[\ell] = (T_\ell\otimes\F_\ell)^F = (T_\ell\otimes\F_\ell)^{F_0}.$$

Moreover, note that $F_0^{2m}=1$, and $\gcd(2m,\ell)=1$ therefore we have a splitting $T_\ell = T_1\oplus T_{\neq 1}$ such that $F_0$ acts as the identity on $T_1$ 
and does not have 1 as an eigenvalue on $T_{\neq 1}$. Note that the eigenvalues of $F_0$ on $T_{\neq 1}$ are non-trivial  $2m$'th roots of unity, and therefore also on 
$T_{\neq 1}\otimes\F_\ell$. Thus, it follows that
$$J_q(\F_{q_0})[\ell]\cong(T_\ell\otimes\F_\ell)^{F_0}=(T_1\otimes\F_\ell) \cong (\Z/\ell\Z)^M,$$ which completes the proof.

\end{proof}

\section{Proof of Theorem \ref{main}}

We can now prove our main theorem. In fact, we prove the following more precise version:

\begin{thm}\label{realmain}

Let $p\neq \ell$ be odd prime numbers, and let $m$ range over positive integers divisible by $\ell-1$ and co-prime to $\ell$, and set $q=p^{2m}$. Finally, let $C_q^-,J_q^-$ denote the quadratic twists of $C_q,J_q$ respectively over $\F_p$. Then

\begin{enumerate}

\item $$\rk J_q(\F_{p^2})[\ell] \sim \frac{q-1}{2m}\cdot \log_\ell(p^2),$$

\item $$1+o_m(1)\leq\frac{\max\left(\rk J_q(\F_p)[\ell],\rk J_q^-(\F_p)[\ell]\right)}{\frac{q-1}{2m}\cdot \log_\ell(p)} \leq 2+o_m(1).$$

\end{enumerate}

\end{thm}

\begin{proof}

For the proof  of (1), note that by lemma \ref{tate} the $\ell$-torsion rank is exactly $a_d$ where $d$ is the order of $q_0^{\frac12}$ modulo $\ell$, which by lemma \ref{adappears}
satisfies $|a_d-\frac{q-1}{m}|\leq \frac{m-1}{m}(q^{\frac12}+1)\leq q^{\frac12}+1$. Since $m=\log_{p^2} q = o(q^{\frac 12})$ the result follows.
\bigskip

For the proof of (2), note that since $\ell$ is odd, we have the isomorphism $$J_q(\F_{p^2})[\ell] \cong J_q(\F_p)[\ell] \oplus J^-_q(\F_p)[\ell],$$ and so the result follows immediately
from (1).

\end{proof}

\end{document}